\documentclass{amsart}
\usepackage{amsmath}
\usepackage{amsfonts}
\usepackage{graphics}
\usepackage{epsfig}
\usepackage{amssymb}
\usepackage{amscd}
\usepackage[all]{xy}
\usepackage{latexsym}
\usepackage{graphicx}
\usepackage{multirow}
\usepackage{geometry}

\newtheorem{thm}{Theorem}[section]
\newtheorem{cor}[thm]{Corollary}
\newtheorem{lem}[thm]{Lemma}
\newtheorem{prop}[thm]{Proposition}

\theoremstyle{definition}
\newtheorem{Def}[thm]{Definition}
\newtheorem{rem}[thm]{Remark}
\newtheorem*{ack}{Acknowledgement}

\numberwithin{equation}{section}
\numberwithin{figure}{section}

\def\trace{{\text{\rm{trace}}}}

\def\rchi{{\hbox{\raise1.5pt\hbox{$\chi$}}}}
\def\Aut{{\text{\rm{Aut}}}}

\def\isom{\cong}
\def\tensor{\otimes}

\def\lam{\lambda}

\newcommand{\Mbar}{{\overline{\mathcal{M}}}}

\newcommand{\bF}{{\mathbb{F}}}
\newcommand{\bL}{{\mathbb{L}}}

\newcommand{\bR}{{\mathbb{R}}}

\newcommand{\bZ}{{\mathbb{Z}}}

\newcommand{\cM}{{\mathcal{M}}}

\newcommand{\cH}{{\mathcal{H}}}

\newcommand{\la}{{\langle}}
\newcommand{\ra}{{\rangle}}
\newcommand{\half}{{\frac{1}{2}}}
\newcommand{\bp}{{\mathbf{p}}}
\newcommand{\bx}{{\mathbf{x}}}

\begin{document}
\large
\setcounter{section}{0}

\title[Recursion for the
Poincar\'e polynomial of the moduli of curves ]
{
Topological recursion for the
Poincar\'e polynomial of the
combinatorial moduli space of curves
}

\author[M.\ Mulase]{Motohico Mulase}
\address{
Department of Mathematics\\
University of California\\
Davis, CA 95616--8633, U.S.A.}
\email{mulase@math.ucdavis.edu}

\author[M.~Penkava]{Michael Penkava}
\address{
Department of Mathematics\\
University of Wisconsin\\
Eau Claire, WI 54702--4004}
\email{penkavmr@uwec.edu}

\begin{abstract}
We show that the  Poincar\'e polynomial associated with
the orbifold cell decomposition of
the moduli space of smooth algebraic curves with
distinct marked points satisfies a topological recursion
formula of the Eynard-Orantin type.
The recursion
uniquely  determines the Poincar\'e polynomials
from the initial data.
Our key discovery is that
the Poincar\'e polynomial is the Laplace transform of the
number of Grothendieck's dessins d'enfants.
\end{abstract}


\maketitle

\allowdisplaybreaks

\tableofcontents

\section{Introduction}
\label{sect:intro}

The Euler characteristic of the moduli space
$\cM_{g,n}$ of smooth algebraic curves
of genus $g$ and $n$ distinct marked points has a closed
formula
\begin{equation}
\label{eq:HZ}
\begin{aligned}
\chi\left(\cM_{g,n}\right)
&=
(-1)^{n-1}\frac{(2g-3+n)!}{(2g-2)!}\cdot \zeta(1-2g)
\\
&=
(-1)^n \frac{(2g-3+n)!}{(2g)!}(2g-1)b_{2g}
\end{aligned}
\end{equation}
due to Harer and Zagier \cite{HZ}, where $\zeta(s)$ is the
Riemann zeta function and $b_r$  the Bernoulli number
defined by
$$
\frac{x}{e^x-1} = \sum_{r=0} \frac{b_r}{r!}\;x^r.
$$
A relation of this formula to quantum field theory,
in particular  matrix models, was discovered
by Penner \cite{P}, and a proof of (\ref{eq:HZ})
in terms of an asymptotic analysis of
the Feynman diagram expansion of the
Penner matrix model was established in
\cite{M1995}.

A Feynman diagram for the Penner model is a
double-edge graph of 't Hooft \cite{tH}, which we call
a \emph{ribbon graph} following Kontsevich \cite{K1992}.
The reason that ribbon graphs appear in the calculation of the
Euler characteristic of the moduli space lies in the
isomorphism of topological orbifolds
\begin{equation}
\label{eq:M=RG}
{\cM}_{g,n}\times \bR_+ ^n \isom RG_{g,n}
\end{equation}
due to Harer \cite{Harer}, Mumford \cite{Mumford}, and Strebel
\cite{Strebel}. Here
\begin{equation}
\label{eq:RG}
RG_{g,n} = \coprod_{\substack{\Gamma {\text{ ribbon graph}}\\
{\text{of type }} (g,n)}}
\frac{\bR_+ ^{e(\Gamma)}}{\Aut (\Gamma)}
\end{equation}
is the smooth orbifold
\cite{STT}
consisting of metric ribbon graphs of a given
topological type $(g,n)$ with valence $3$ or more,
$e(\Gamma)$ is the number of edges of the ribbon graph $\Gamma$,
and $\Aut(\Gamma)$ is the group of ribbon graph automorphisms
of $\Gamma$ that fix every face.
The Penner model is the generating function of the
Euler characteristic of $RG_{g,n}$. As an element
 of the formal power series in two variables
$z$ and $M$, we have the equality
\begin{multline}
\label{eq:matrix}
\log \int_{\cH_M}\exp
\left(
-\sum_{j=2} ^\infty \frac{(\sqrt{z})^{j-2}}{j}\;
\trace (X^j)
\right)
dX
\\
=
\sum_{\substack{g\ge 0, \;n>0\\2g-2+n>0}}
(-1)^n\;
\chi\left(RG_{g,n}\right) \frac{M^n}{n!}\; z^{2g-2+n},
\end{multline}
where the parameter $M$ appears as the size of
the Hermitian matrix $X$ in the left-hand side,
$\cH_M$ is the linear
space of $M\times M$ Hermitian matrices, and $dX$
is a suitably normalized Lebesgue measure on
$\cH_M$. We refer to \cite{M1995} for the precise
meaning of the equality.

Although the matrix integral (\ref{eq:matrix}) gives an effective
tool to calculate the Euler characteristic
$$
\chi\left(RG_{g,n}\right)  =
\sum_{\substack{\Gamma {\text{ ribbon graph}}\\
{\text{of type }} (g,n)}}
\frac{(-1) ^{e(\Gamma)}}{|\Aut (\Gamma)|},
$$
it does not tell us anything about more refined information of
the orbifold cell structure of $RG_{g,n}$.
One can ask:
\emph{Isn't there any effective tool to find more
numerical information about
the orbifold} $RG_{g,n}$?

The purpose of this paper is to answer this question.
Our answer is again based on an idea from physics,
this time utilizing
the Eynard-Orantin \emph{topological recursion theory}
\cite{EO1}.

For a fixed $(g,n)$ in the stable range, {\it i.e.}, $2g-2+n>0$, we
choose $n$ variables $t_1,t_2,\dots,t_n$, and
define the function
$$
z(t_i,t_j) = \frac{(t_i+1)(t_j+1)}{2(t_i+t_j)}.
$$
An edge $\eta$ of a ribbon graph $\Gamma$ bounds two
faces, say $i_\eta$ and $j_\eta$. These two faces may be
actually the same. Now we define the
\emph{Poincar\'e polynomial} of $RG_{g,n}$ in the $z$-variables
by
\begin{equation}
\label{eq:Poincare}
F_{g,n}(t_1,\dots,t_n) =
\sum_{\substack{\Gamma {\text{ ribbon graph}}\\
{\text{of type }} (g,n)}}
\frac{(-1)^{e(\Gamma)}}{|\Aut(\Gamma)|}
\prod_{\substack{\eta \text{ edge}\\
\text{of }\Gamma}}
z\!\left(t_{i_\eta},t_{j_\eta}
\right),
\end{equation}
which is a polynomial in $z(t_i,t_j)$ but actually a
symmetric rational function in $t_1,\dots,t_n$.
Our main theorem of this paper is a
topological recursion formula that uniquely determines
the Poincar\'e polynomials. To state the formula
in a compact fashion, we use the following notation.
Let $N=\{1,2,\dots,n\}$ be the index set labeling the
marked points of a smooth algebraic curve. The faces
of a ribbon graph of type $(g,n)$
are also labeled by the same set. For every
subset $I\subset N$, we denote
$$
t_I = \left(t_i\right)_{i\in I}.
$$

\begin{thm}
\label{thm:main}
The Poincar\'e polynomial $F_{g,n}(t_N)$
with $(g,n)$ in the stable range
$$
2g-2+n>0
$$
is uniquely determined by the following
topological recursion formula from the
initial values $F_{0,3}(t_1,t_2,t_3)$ and
$F_{1,1}(t_1)$.
\begin{multline}
\label{eq:main}
F_{g,n}(t_N)
\\
=
-\frac{1}{16}
\int_{-1} ^{t_1}
\Biggr[
\sum_{j=2} ^n
\frac{t_j}{t^2-t_j^2}
\Bigg(
\frac{(t^2-1)^3}{t^2}\frac{\partial}{\partial t}
F_{g,n-1}(t,t_{N\setminus\{1,j\}})
-
\frac{(t_j^2-1)^3}{t_j^2}\frac{\partial}{\partial t_j}
F_{g,n-1}(t_{N\setminus\{1\}})
\Bigg)
\\
+
\sum_{j=2} ^n
\frac{(t^2-1)^2}{t^2}\frac{\partial}{\partial t}
F_{g,n-1}(t,t_{N\setminus\{1,j\}})
\\
+
\frac{1}{2}\;\frac{(t^2-1)^3}{t^2}
\frac{\partial^2}{\partial u_1\partial u_2}
\Bigg(
F_{g-1,n+1}(u_1,u_2,t_{N\setminus\{1\}})
\\
+
\sum_{\substack{g_1+g_2=g\\I\sqcup J=N\setminus\{1\}}}
^{\rm{stable}}
F_{g_1,|I|+1}(u_1,t_I)F_{g_2,|J|+1}(u_2,t_J)
\Bigg)
\Bigg|_{u_1=u_2=t}\;
\Biggr]\;dt.
\end{multline}
Here the last sum is taken over all partitions $g_1+g_2 = g$ and
set partitions $I\sqcup J=N\setminus\{1\}$ subject to
the stability conditions $2g_1-1+|I|>0$ and $2g_2-1+|J|>0$.
\end{thm}

\begin{rem}
\begin{enumerate}
\item
It was shown in \cite{CMS} that
the symmetric differential
$d_1\tensor \cdots\tensor d_n F_{g,n}(t_N)$
satisfies an Eynard-Orantin type topological recursion.
The relation between the Euler characteristic
of $\cM_{g,n}$ and the Eynard-Orantin theory was
first pointed out in \cite{N2}.

\item
The word \emph{topological} recursion refers to the
inductive structure on the quantity $2g-2+n$, which is
the absolute value of the Euler characteristic of an
oriented
$n$-punctured surface of genus $g$. Reduction of the quantity
$2g-2+n$ by one has appeared in many recent works on
moduli theory of curves,
Gromov-Witten theory and related topics. This includes
the operation of
 cutting off a pair of
pants from a bordered surface
as in \cite{Mir1, Mir2}, the Hurwitz move or the cut-and-join
equation of Hurwitz numbers \cite{GJ, H, V}, the edge
removal operation on $RG_{g,n}$ of \cite{CMS, N1},
and many generalizations including
\cite{BKMP, BM, EMS, LX, M, MS, MZ,
Z1, Z2}.
\end{enumerate}
\end{rem}

By the definition of $F_{g,n}(t_N)$ and the fact that
$z(1,1)=1$, the Poincar\'e polynomial recovers the
Euler characteristic of the moduli space
$\cM_{g,n}$ as the special value
$$
F_{g,n}(1,1,\dots,1) = \chi\left(RG_{g,n}\right)
=(-1)^n\chi\left(\cM_{g,n}\right).
$$
The Poincar\'e polynomial
 becomes particularly simple when $n=1$. We have
\begin{equation}
\label{eq:Fg1intro}
F_{g,1}(t) =
\sum_{\substack{\Gamma {\text{ ribbon graph}}\\
{\text{of type }} (g,1)}}
\frac{(-1)^{e(\Gamma)}}{|\Aut(\Gamma)|}\;
z^{e(\Gamma)},
\end{equation}
where
\begin{equation}
\label{eq:ztt}
z=z(t,t) = \frac{(t+1)^2}{4t}.
\end{equation}
An immediate generalization of the above formula is 
the diagonal value
\begin{equation}
\label{eq:tj=t}
F_{g,n}(t, t, \dots,t) =
\sum_{\substack{\Gamma {\text{ ribbon graph}}\\
{\text{of type }} (g,n)}}
\frac{(-1)^{e(\Gamma)}}{|\Aut(\Gamma)|}\;
z^{e(\Gamma)}.
\end{equation}
Because of this formula our terminology of 
calling $F_{g,n}(t_N)$ the
``Poincar\'e polynomial'' is justified.

Although it is not obvious from the definition or even
from Theorem~\ref{thm:main}, the symmetric rational function
$F_{g,n}(t_1,\dots,t_n)$ is actually a \emph{Laurent polynomial}.
Therefore, it makes sense to extract the highest degree terms.
If we naively extract the top degree term from $z(t_i,t_j)$, then 
we obtain
$$
z^{\rm{top}}(t_i,t_j)=\frac{t_it_j}{2(t_i+t_j)}.
$$
Since the number of edges of a ribbon graph is maximum 
for a trivalent graph, we obtain the following.

\begin{thm}
\label{thm:appl}
The Poincar\'e polynomial $F_{g,n}(t_N)$ is a
Laurent polynomial in $t_1,t_2,\dots,t_n$
of degree $3(2g-2+n)$
such that every monomial term contains
only an odd  power of each $t_j$. The leading homogeneous
polynomial
$F_{g,n}^{\rm{top}}(t_N)$ of $F_{g,n}(t_N)$
is given by
\begin{equation}
\label{eq:Ftop}
\begin{aligned}
F_{g,n}^{\rm{top}}(t_N)
&=
\sum_{\substack{\Gamma {\text{ trivalent ribbon}}\\
{\text{graph of type }} (g,n)}}
\frac{(-1)^{e(\Gamma)}}{|\Aut(\Gamma)|}
\prod_{\substack{\eta \text{ edge}\\
\text{of }\Gamma}}
\frac{t_{i_\eta} t_{j_\eta}}{2\left(
t_{i_\eta}+ t_{j_\eta}
\right)}
\\
&=
\frac{(-1)^n}{2^{5g-5+2n}}
\sum_{\substack{d_1+\cdots+d_n\\=3g-3+n}}
\la \tau_{d_1}\cdots\tau_{d_n}\ra_{g,n}
\prod_{j=1}^n \frac{(2d_j)!}{d_j!} \left(
\frac{t_j}{2}
\right)^{2d_j+1},
\end{aligned}
\end{equation}
where
$$
\la \tau_{d_1}\cdots\tau_{d_n}\ra_{g,n}
=
\int_{\Mbar_{g,n}}c_1(\bL_1)^{d_1}\cdots c_1(\bL_n)^{d_n}
$$
are the $\psi$-class intersection numbers of the tautological
cotangent line bundles $\bL_1,\dots,\bL_n$ on $\Mbar_{g,n}$.
The above formula is identical to the boxed formula
of Kontsevich \cite[page 10]{K1992}.
The topological recursion
{\rm{(\ref{eq:main})}} restricts to the leading
 terms $F_{g,n}^{\rm{top}}(t_N)$ and recovers
 the Virasoro constraint condition, or the DVV-formula,
of the $\psi$-class intersection numbers due to
Dijkgraaf-Verlinde-Verlinde \cite{DVV} and Witten
\cite{W1991}.
\end{thm}

It requires the deep theory of Mirzakhani \cite{Mir1, Mir2}
to relate the leading terms $F_{g,n}^{\rm{top}}(t_N)$ and the
intersection numbers
because of the difference between $\cM_{g,n}$ and
$\Mbar_{g,n}$.
The contribution of
Theorem~\ref{thm:appl} is to identify the origin of the
Virasoro constraint condition as the edge-removal
operation of ribbon graphs of
 \cite{CMS, N1}, and to  clarify the relation between the
 combinatorics of counting problems 
 and the geometry of intersection numbers.
 For the moduli space of vector bundles on curves,
 Harder and Narasimhan used Deligne's solution of the
 Weil conjecture to obtain the Poincar\'e polynomial. 
 Although what we are dealing with in this article is
 much simpler than the situation
 of \cite{HN}, we  find that again a counting problem
 plays a key role in calculating the Poincar\'e polynomial. 
 Here the critical differences are that we use lattice point counting 
 rather than moduli theory over the finite field
 $\bF_q$, and that through (\ref{eq:Ftop}) 
 the counting problem also leads to the intersection numbers
 of the \emph{compactified} moduli space $\Mbar_{g,n}$.

 We note that the polynomial situation 
 of Theorem~\ref{thm:appl} is
similar to the case of simple Hurwitz numbers
studied in \cite{EMS, MZ}.
Indeed, the result of \cite{MZ} is that the
\emph{Laplace
transform} of simple Hurwitz numbers as a function of
a partition is a \emph{polynomial} that satisfies a
topological recursion. This recursion proves the DVV formula
of \cite{DVV, K1992, W1991} when restricted to the
leading terms, and also proves the $\lambda_g$-conjecture
(the theorem of \cite{FP1, FP2})
when restricted to the lowest degree homogeneous terms.
In a surprising similarity, we show that the Laurent polynomial
$F_{g,n}(t_1,\dots,t_n)$ is the Laplace transform of
the number of Grothendieck's \emph{dessins d'enfants}
\cite{Belyi, MP1998, SL}.

One can ask: \emph{Why does the Laplace transform appear
in this context?} A short answer is that the Laplace transform
here is in fact the \emph{mirror map} that transforms the
A-model side of topological string theory to the B-model side.
We do not investigate this idea any further in this paper, and
refer to the introduction of \cite{CMS, EMS, MZ}
for more discussion.

This paper is organized as follows. We review the necessary
information on the ribbon graph complex in
Section~\ref{sect:combinatorial}. In Section~\ref{sect:lattice},
we recall the topological recursion for the number of lattice
points of $RG_{g,n}$ that was established in \cite{CMS}.
We then show in
Section~\ref{sect:LT} that the Laplace transform of this number
 is
exactly the Poincar\'e polynomial
of (\ref{eq:Poincare}). A differential equation
for the Poincar\'e polynomials is derived
in Section~\ref{sect:recursion}. The initial values of the
recursion formula are calculated in
Section~\ref{sect:initial}. In the final section we prove
Theorem~\ref{thm:main} and
Theorem~\ref{thm:appl}.

\section{The combinatorial model of the moduli space}
\label{sect:combinatorial}

We begin by listing basic facts about ribbon graphs
and the combinatorial model for the moduli space
$\cM_{g,n}$ due to Harer
\cite{Harer},
Mumford \cite{Mumford} and Strebel
\cite{Strebel},  following  \cite{MP1998}.
Ribbon graphs are often referred to as
 Grothendieck's
\emph{dessins d'enphants}. The standard literature on this
subject
is \cite{SL}, which contains Grothendieck's \emph{esquisse}.
We do not consider any number theoretic aspects
of the dessins in this paper.

A \emph{ribbon graph} of topological type $(g,n)$
is the $1$-skeleton of a cell-decomposition of a closed
oriented topological surface $\Sigma$ of genus $g$
that decomposes the surface into a disjoint union of
$v$ $0$-cells, $e$ $1$-cells, and $n$ $2$-cells. The Euler
characteristic of the surface is given by $2-2g = v-e+n$.
The $1$-skeleton of a cell-decomposition is a graph
$\Gamma$ drawn on $\Sigma$, which consists
of $v$ vertices and $e$ edges.
An edge can form a loop. We denote by
$\Sigma_\Gamma$ the cell-decomposed surface with $\Gamma$
its $1$-skeleton.
Alternatively, a ribbon graph can be
defined as a graph with a cyclic order
given to the incident half-edges at each vertex. By abuse of
terminology, we call the boundary of a $2$-cell of
$\Sigma_\Gamma$ a
\emph{boundary}  of $\Gamma$, and the $2$-cell itself as
a \emph{face} of $\Gamma$.

A \emph{metric} ribbon graph is a ribbon graph with a
positive real number (the length) assigned to each edge.
For a given ribbon graph $\Gamma$ with $e=e(\Gamma)$
edges, the space of metric
ribbon graphs is $\bR_+ ^{e(\Gamma)}/\Aut (\Gamma)$,
where the
automorphism group acts by permutations of edges
(see \cite[Section~1]{MP1998}).
We restrict ourselves to the case that
$\Aut (\Gamma)$ fixes each $2$-cell of the cell-decomposition.
We also require that every vertex of a ribbon graph has degree
(i.e., valence) $3$ or more.
Using the canonical holomorphic coordinate systems on
a topological surface of \cite[Section~4]{MP1998}
 and the Strebel differentials \cite{Strebel}, we have
an isomorphism of topological orbifolds \cite{Harer,Mumford}
\begin{equation}
\label{eq:M=RG}
{\cM}_{g,n}\times \bR_+ ^n \isom RG_{g,n}.
\end{equation}
Here
$$
RG_{g,n} = \coprod_{\substack{\Gamma {\text{ ribbon graph}}\\
{\text{of type }} (g,n)}}
\frac{\bR_+ ^{e(\Gamma)}}{\Aut (\Gamma)}
$$
is the orbifold consisting of metric ribbon graphs of a given
topological type $(g,n)$.
The gluing of orbi-cells
 is done by making
the length of a non-loop edge tend to $0$. The space
 $RG_{g,n}$ is a smooth orbifold
(see \cite[Section~3]{MP1998} and \cite{STT}).
We denote by $\pi:RG_{g,n}\longrightarrow
\bR_+ ^n$ the natural projection via (\ref{eq:M=RG}), which
is the assignment of the collection of perimeter length
of each boundary to a given metric ribbon graph.

Take a ribbon graph $\Gamma$. Since $\Aut(\Gamma)$
fixes every boundary component of $\Gamma$, they are labeled
by $N=\{1,2\dots,n\}$. For the moment let us give a label to
each edge of $\Gamma$ by an index set $E = \{1,2,\dots,e\}$.
The edge-face incidence matrix  is defined by
\begin{equation}
\label{eq:incidence}
\begin{aligned}
A_\Gamma &= \big[
a_{i\eta}\big]_{i\in N,\;\eta\in E};\\
a_{i\eta} &= \text{ the number of times edge $\eta$ appears in
face $i$}.
\end{aligned}
\end{equation}
Thus $a_{i\eta} = 0, 1,$ or $2$, and the sum of the
entries in each column is
always $2$. The $\Gamma$ contribution of the space
$\pi^{-1}(p_1,\dots,p_n) = RG_{g,n}(\bp)$
 of metric ribbon graphs with
a prescribed perimeter $\mathbf{p}=(p_1,\dots,p_n)$ is the orbifold
polytope
$$
P_\Gamma (\mathbf{p})/\Aut(\Gamma),\qquad
P_\Gamma (\mathbf{p})= \{\mathbf{x}\in \bR_+ ^e\;|\;
A_\Gamma \mathbf{x} = \mathbf{p}\},
$$
where $\mathbf{x}=(\ell_1,\dots,\ell_e)$ is the collection of
edge lengths of the metric ribbon graph $\Gamma$. We have
\begin{equation}
\label{eq:sump}
\sum_{i\in N} p_i= \sum_{i\in N}
\sum_{\eta\in E}a_{i\eta}\ell_\eta =
2\sum_{\eta\in E}
\ell_\eta.
\end{equation}

\section{Topological recursion for the number of integral ribbon
graphs}
\label{sect:lattice}

In this section we recall the topological recursion for
the number of
 metric ribbon graphs $RG_{g,n} ^{\bZ_+}$
whose edges have
integer lengths, following \cite{CMS}.
We call such a ribbon graph
an \emph{integral ribbon graph}.
We can interpret an integral ribbon graph as
Grothendieck's \emph{dessin d'enfant} by considering
an edge of integer length as a chain of edges of length one
connected by bivalent vertices, and reinterpreting
 the notion of
$\Aut(\Gamma)$ suitably.
Since we do not go into the number theoretic aspects of
dessins, we stick to the more geometric
notion of integral ribbon graphs.

\begin{Def}
\label{def:N}
The weighted number $\big| RG_{g,n} ^{\bZ_+}(\bp)\big|$
of integral ribbon graphs with
prescribed perimeter lengths
$\bp\in\bZ_+ ^n$ is defined by
\begin{equation}
\label{eq:Ngn}
N_{g,n}(\bp) =
\big| RG_{g,n} ^{\bZ_+}(\bp)\big|
=\sum_{\substack{\Gamma {\text{ ribbon graph}}\\
{\text{of type }} (g,n)}}
\frac{\big|\{\bx\in \bZ_+ ^{e(\Gamma)}\;|\;A_\Gamma \bx = \bp\}
\big|}{|\Aut(\Gamma)|}.
\end{equation}
\end{Def}

Since the finite set
$\{\bx\in \bZ_+ ^{e(\Gamma)}\;|\;A_\Gamma \bx = \bp\}$
is a collection of lattice points in the polytope $P_\Gamma(\bp)$
with respect to the canonical integral structure $\bZ\subset
\bR$ of the real numbers, $N_{g,n}(\bp)$ can be thought of
counting the
number of \emph{lattice points}
in $RG_{g,n}(\bp)$ with a weight factor
$1/|\Aut(\Gamma)|$ for each ribbon graph.
The function $N_{g,n}(\bp)$ is a symmetric function in
$\bp = (p_1,\dots,p_n)$
because the summation runs over all ribbon graphs of topological
type $(g,n)$.

\begin{rem}
Since the integral vector $\bx$ is restricted to
take strictly positive values,
we would have $N_{g,n}(\bp) = 0$ if we were to
substitute $\bp=0$.
This normalization is natural from the point of view of
lattice point counting and Grothendieck's
\emph{dessins d'enphants}. However, we do not
make such a substitution in this paper because we
consider $\bp$ as a strictly positive integer vector.
This situation is similar to Hurwitz theory
 \cite{EMS, MZ}, where  a partition $\mu$ is
 a strictly positive integer vector that plays the role of our
 $\bp$. We note that  a different
 assignment of values was suggested in \cite{N1, N2}.
\end{rem}

For brevity of notation, we denote by $p_I = (p_i)_{i\in I}$
for a subset $I\in N=\{1,2\dots,n\}$. The cardinality of $I$ is
denoted by $|I|$. The following topological recursion
formula was proved in \cite{CMS} using the idea of
ciliation of a ribbon graph.

\begin{thm}[\cite{CMS}]
\label{thm:integralrecursion}
The number of integral ribbon graphs
with prescribed boundary lengthes satisfies the
topological recursion formula
\begin{multline}
\label{eq:integralrecursion}
p_1  N_{g,n}(p_N)
=
\half
\sum_{j=2} ^n
\Bigg[
\sum_{q=0} ^{p_1+p_j}
q(p_1+p_j-q)  N_{g,n-1}(q,p_{N\setminus\{1,j\}})
\\
+
H(p_1-p_j)\sum_{q=0} ^{p_1-p_j}
q(p_1-p_j-q)N_{g,n-1}(q,p_{N\setminus\{1,j\}})
\\
-
H(p_j-p_1)\sum_{q=0} ^{p_j-p_1}
q(p_j-p_1-q)
N_{g,n-1}(q,p_{N\setminus\{1,j\}})
\Bigg]
\\
+\half \sum_{0\le q_1+q_2\le p_1}q_1q_2(p_1-q_1-q_2)
\Bigg[
N_{g-1,n+1}(q_1,q_2,p_{N\setminus\{1\}})
\\
+\sum_{\substack{g_1+g_2=g\\
I\sqcup J=N\setminus\{1\}}} ^{\rm{stable}}
N_{g_1,|I|+1}(q_1,p_I)
N_{g_2,|J|+1}(q_2,p_J)\Bigg].
\end{multline}
Here
$$
H(x) = \begin{cases}
1 \qquad x>0\\
0 \qquad x\le 0
\end{cases}
$$
is the Heaviside function,
and the last sum is taken for all partitions
$g=g_1+g_2$ and $I\sqcup J=N\setminus \{1\}$
 subject to the stability conditions
$2g_1-1+{I}>0$ and $2g_2-1+|J|>0$.
\end{thm}

\section{The Laplace transform of the number of
integral ribbon graphs}
\label{sect:LT}

Let us consider the \emph{Laplace transform}
 \begin{equation}
 \label{eq:Lgn}
 L_{g,n}(w_1,\dots,w_n) \overset{\rm{def}}{=}
 \sum_{\bp\in\bZ_{+} ^n} N_{g,n}(\bp) e^{-\la \bp,w\ra}
 \end{equation}
 of the number of
integral ribbon graphs $N_{g,n}(\bp)$,
where $\la \bp,w\ra=p_1w_1+\cdots+p_nw_n$, and
the summation is taken over all integer vectors $\bp\in\bZ_+^n$
of strictly positive entries.
In this section we prove that after the coordinate change
of \cite{CMS}
from the $w$-coordinates to the $t$-coordinates
defined by
\begin{equation}
\label{eq:t}
e^{-w_j} = \frac{t_j+1}{t_j-1}, \qquad j=1,2,\dots,n,
\end{equation}
 the Laplace transform
$L_{g,n}(w_N)$ becomes the Poincar\'e
polynomial
\begin{equation}
\label{eq:F}
F_{g,n}(t_1,\dots,t_n) =
L_{g,n}\big(w_1(t),\dots,w_n(t)\big).
\end{equation}

The Laplace transform $L_{g,n}(w_N)$ can be evaluated
using the definition
of the number of integral ribbon graphs
(\ref{eq:Ngn}).
Let $a_\eta$ be the $\eta$-th column of the
incidence matrix $A_\Gamma$ so that
\begin{equation}
\label{eq:Agammacolumn}
A_\Gamma = \big[a_1\big|a_2\big|\cdots\big|a_{e(\Gamma)}\big].
\end{equation}
Then

\begin{multline}
\label{eq:LinA}
L_{g,n}(w_N) =
\sum_{\bp\in\bZ_+^n}
N_{g,n}(\bp) e^{-\la \bp,w\ra}
\\
=
\sum_{\substack{\Gamma {\text{ ribbon graph}}\\
{\text{of type }} (g,n)}}
\sum_{\bp\in\bZ_+^n}
\frac{1}{|\Aut(\Gamma)|}
\big|\{\bx\in \bZ_+ ^{e(\Gamma)}\;|\;A_\Gamma \bx = \bp\}
\big|
 e^{-\la \bp,w\ra}
 \\
 =
 \sum_{\substack{\Gamma {\text{ ribbon graph}}\\
{\text{of type }} (g,n)}}
\frac{1}{|\Aut(\Gamma)|}
\sum_{\bx\in\bZ_+^{e(\Gamma)}}
e^{-\la A_\Gamma \bx,w\ra}
\\
 =
 \sum_{\substack{\Gamma {\text{ ribbon graph}}\\
{\text{of type }} (g,n)}}
\frac{1}{|\Aut(\Gamma)|}
\prod_{\substack{\eta \text{ edge}\\
\text{of }\Gamma}}
\sum_{\ell_\eta=1}^\infty
e^{-\la a_\eta,w\ra\ell_\eta}
\\
 =
 \sum_{\substack{\Gamma {\text{ ribbon graph}}\\
{\text{of type }} (g,n)}}
\frac{1}{|\Aut(\Gamma)|}
\prod_{\substack{\eta \text{ edge}\\
\text{of }\Gamma}}
\frac{
e^{-\la a_\eta,w\ra}}
{1-e^{-\la a_\eta,w\ra}}.
\end{multline}
Every edge $\eta$ bounds two faces, which we call face
$i_\eta ^+$ and face $i_\eta ^-$. When $a_{i\eta}=2$,
these faces are the same.
We then calculate
\begin{equation}
\label{eq:aw}
\frac{
e^{-\la a_\eta,w\ra}}
{1-e^{-\la a_\eta,w\ra}}
=
-z\!\left(t_{i_\eta ^+},t_{i_\eta ^-}
\right),
\end{equation}
where
\begin{equation}
\label{eq:z}
z(t_i,t_j) \overset{\rm{def}}{=}
\frac{(t_i+1)(t_j+1)}{2(t_i+t_j)}.
\end{equation}
This follows from (\ref{eq:t}) and
\begin{align*}
\frac{e^{-(w_i+w_j)}}{1-e^{-(w_i+w_j)}}
&=
\frac{\frac{(t_i+1)(t_j+1)}{(t_i-1)(t_j-1)}}
{1-\frac{(t_i+1)(t_j+1)}{(t_i-1)(t_j-1)}}
=
-\frac{(t_i+1)(t_j+1)}{2(t_i+t_j)}
=-z(t_i,t_j),
\\
\frac{e^{-2w_i}}{1-e^{-2w_i}}
&=
-\frac{(t_i+1)^2}{4t_i}
=-z(t_i,t_i).
\end{align*}
Note that since $z(t_i,t_j)$ is a symmetric
function, which face is
named $i_\eta ^+$ or $i_\eta ^-$ does not matter.
From (\ref{eq:LinA}) and (\ref{eq:aw}), we have
established
\begin{thm}
\label{thm:Finz}
The Laplace transform $L_{g,n}(w_N)$ in terms
of the $t$-coordinates {\rm{(\ref{eq:t})}}
 is the Poincar\'e polynomial
\begin{equation}
\label{eq:Finz}
F_{g,n}(t_N) =
 \sum_{\substack{\Gamma {\text{ ribbon graph}}\\
{\text{of type }} (g,n)}}
\frac{(-1)^{e(\Gamma)}}{|\Aut(\Gamma)|}
\prod_{\substack{\eta \text{ edge}\\
\text{of }\Gamma}}
z\!\left(t_{i_\eta ^+},t_{i_\eta ^-}
\right).
\end{equation}
\end{thm}

\begin{cor}
\label{cor:F1}
The evaluation of $F_{g,n}(t_N)$ at
$t_1=\cdots=t_n=1$ gives the
Euler characteristic of $RG_{g,n}$
\begin{equation}
\label{eq:euler}
F_{g,n}(1,1\dots,1) = \chi\left(RG_{g,n}\right)
=(-1)^n\chi\left(\cM_{g,n}\right).
\end{equation}
Furthermore, if we evaluate at $t_j=-1$ for any $j$, then we have
\begin{equation}
\label{eq:Fj}
F_{g,n}(t_1,t_2\dots,t_n)\big|_{t_j=-1} = 0
\end{equation}
as a function in the rest of the variables $t_{N\setminus\{j\}}$.
\end{cor}

\begin{proof}
The Euler characteristic calculation immediately follows
from $z(1,1) = 1$.

Consider a ribbon graph $\Gamma$ of type $(g,n)$. Its $j$-th
face
 has at least one edge on its boundary. Therefore,
$$
\prod_{\eta \text{ edge of }\Gamma}
z\!\left(t_{i_\eta ^+},t_{i_\eta ^-}
\right)
$$
has a factor $(t_j+1)$ by (\ref{eq:z}). It holds for every ribbon
graph $\Gamma$ in the summation of (\ref{eq:Finz}).
Therefore, (\ref{eq:Fj}) follows.
\end{proof}

\section{Topological recursion for the Poincar\'e
polynomials}
\label{sect:recursion}

In this section we prove that
the Poincar\'e polynomials satisfy a differential equation.

\begin{thm}
\label{thm:LTofN}
The Poincar\'e polynomial $F_{g,n}(t_N)$ satisfies the following
differential recursion equation.
\begin{multline}
\label{eq:LTofN}
\frac{\partial}{\partial t_1}F_{g,n}(t_N)
\\
=
-\frac{1}{16}
\sum_{j=2} ^n
\left[\frac{t_j}{t_1^2-t_j^2}
\left(
\frac{(t_1^2-1)^3}{t_1^2}\frac{\partial}{\partial t_1}
F_{g,n-1}(t_{N\setminus\{j\}})
-
\frac{(t_j^2-1)^3}{t_j^2}\frac{\partial}{\partial t_j}
F_{g,n-1}(t_{N\setminus\{1\}})
\right)
\right]
\\
-\frac{1}{16}
\sum_{j=2} ^n
\frac{(t_1^2-1)^2}{t_1^2}\frac{\partial}{\partial t_1}
F_{g,n-1}(t_{N\setminus\{j\}})
\\
-
\frac{1}{32}\;\frac{(t_1^2-1)^3}{t_1^2}
\frac{\partial^2}{\partial u_1\partial u_2}
\Bigg[
F_{g-1,n+1}(u_1,u_2,t_{N\setminus\{1\}})
\\
+
\sum_{\substack{g_1+g_2=g\\I\sqcup J=N\setminus\{1\}}}
^{\rm{stable}}
F_{g_1,|I|+1}(u_1,t_I)F_{g_2,|J|+1}(u_2,t_J)
\Bigg]
\Bigg|_{u_1=u_2=t_1}.
\end{multline}
\end{thm}

\begin{proof}
We first calculate the Laplace transform of
(\ref{eq:integralrecursion}) and establish a
differential equation  for $L_{g,n}(w_N)$. We then change the
variables from $w_N$ to $t_N$ using (\ref{eq:t}).
The operation we need to do is to multiply
 both sides of (\ref{eq:integralrecursion}) by $e^{-\la \bp,w\ra}$
and take the sum with respect to all integers $p_1\ge 0$
and $p_{N\setminus\{1\}}\in \bZ_+^{n-1}$.
Since the left-hand side of (\ref{eq:integralrecursion}) is
$p_1N_{g,n}(p_N)$, we can allow $p_1=0$ in the
summation.

The result of this operation to the left-hand side of
(\ref{eq:integralrecursion}) is $-\frac{\partial}{\partial w_1}
L_{g,n}(w_N)$. The operation applied to
the  first line of the right-hand side gives
\begin{multline}
\label{eq:1st}
\sum_{j=2} ^n
\sum_{p_1=0}^\infty
\sum_{p_{N\setminus\{1\}}\in\bZ_{+}^{n-1}}
\sum_{q=0} ^{p_1+p_j} q\;\frac{p_1+p_j-q}{2}\;
N_{g,n-1}(q,p_{N\setminus\{1,j\}})
e^{-\la \bp,w\ra}
\\
=
\sum_{j=2} ^n \sum_{q=0} ^\infty
\sum_{p_{N\setminus\{1,j\}}\in\bZ_{+} ^{n-2}}
q\;N_{g,n-1}(q,p_{N\setminus\{1,j\}})
e^{-\la p_{N\setminus\{1,j\}}, w_{N\setminus\{1,j\}}\ra}
e^{-qw_1}
\\
\times
\sum_{\ell = 0} ^\infty
\ell e^{-2\ell w_1} \sum_{p_j=1} ^{q+2\ell} e^{p_j(w_1-w_j)},
\end{multline}
where
we set $p_1+p_j-q=2\ell$. Note that $N_{g,n}(p_N)=0$
unless $p_1+\cdots+p_n$ is even, because of
(\ref{eq:sump}). Therefore, in the Laplace
transform we are summing
over all $p_N\in\bZ_{+}^n$ such that $p_1+\cdots+p_n\equiv 0
\mod 2$.
Since
$N_{g,n-1}(q,p_{N\setminus\{1,j\}})=0$ unless
$q+p_2+\cdots +\widehat{p_j}+\cdots +p_n\equiv 0 \mod 2$,
only those $p_1, p_j$ and $q$ satisfying
$p_1+p_j-q\equiv 0\mod 2$ contribute in the summation.
Thus we can replace $p_1+p_j-q$ by $2\ell$.
The  $p_j$-summation of (\ref{eq:1st}) gives
\begin{multline*}
\sum_{p_j=1} ^{q+2\ell}e^{-qw_1}\ell e^{-2\ell w_1}
  e^{p_j(w_1-w_j)}
=
\ell e^{-(q+2\ell) w_1}
\frac{e^{w_1-w_j}- e^{(1+q+2\ell)(w_1-w_j)}}
{1-e^{w_1-w_j}}
\\
=
\frac{e^{w_1-w_j}}{1-e^{w_1-w_j}}
\left(
e^{-qw_1}\ell e^{-2\ell w_1}
-e^{-qw_j}\ell e^{-2\ell w_j}
\right).
\end{multline*}
Since the $\ell$-summation and the $q$-summation are separated
now,
(\ref{eq:1st}) becomes
\begin{multline*}
\sum_{j=2}^n
\frac{e^{w_1-w_j}}{1-e^{w_1-w_j}}
\Bigg[
\frac{1}{(e^{w_1}-e^{-w_1})^2}
\left(
-\frac{\partial}{\partial w_1}
\right)
L_{g,n-1}(w_{N\setminus\{j\}})
\\
-
\frac{1}{(e^{w_j}-e^{-w_j})^2}
\left(
-\frac{\partial}{\partial w_j}
\right)
L_{g,n-1}(w_{N\setminus\{1\}})
\Bigg].
\end{multline*}
The second line of (\ref{eq:integralrecursion}) gives
\begin{multline*}
\sum_{j=2} ^n
\sum_{p_1=0}^\infty
\sum_{p_{N\setminus\{1\}}\in\bZ_{+}^{n-1}}
H(p_1-p_j)
\sum_{q=0} ^{p_1-p_j} q\;\frac{p_1-p_j-q}{2}
N_{g,n-1}(q,p_{N\setminus\{1,j\}})
e^{-\la \bp,w\ra}
\\
=
\sum_{j=2} ^n
\sum_{\ell=0} ^\infty \ell e^{-2\ell w_1}
\sum_{p_j=1} ^\infty
 e^{-p_j(w_1+w_j)}
\\
\times
\sum_{q=0} ^\infty e^{-qw_1}
\sum_{p_{N\setminus\{1,j\}}\in\bZ_{+} ^{n-2}}
q\;N_{g,n-1}(q,p_{N\setminus\{1,j\}})
e^{-\la p_{N\setminus\{1,j\}},w_{N\setminus\{1,j\}}\ra}
\\
=
\sum_{j=2} ^n
\frac{e^{-(w_1+w_j)}}{1-e^{-(w_1+w_j)}}\;
\frac{1}{(e^{w_1}-e^{-w_1})^2}
\left(
-\frac{\partial}{\partial w_1}
\right)
L_{g,n-1}(w_{N\setminus\{j\}})
\\
=
\sum_{j=2} ^n
\left(
\frac{1}{1-e^{-(w_1+w_j)}}-1
\right)
\frac{1}{(e^{w_1}-e^{-w_1})^2}
\left(
-\frac{\partial}{\partial w_1}
\right)
L_{g,n-1}(w_{N\setminus\{j\}}),
\end{multline*}
where we set $p_1-p_j-q = 2\ell$.
Similarly, after putting
$p_j-p_1-q=2\ell$, the third line of
(\ref{eq:integralrecursion}) yields
\begin{multline*}
-\sum_{j=2} ^n
\sum_{p_1=0}^\infty
\sum_{p_{N\setminus\{1\}}\in\bZ_{+}^{n-1}}
 H(p_j-p_1)
\sum_{q=0} ^{p_j-p_1} q\;\frac{p_j-p_1-q}{2}
N_{g,n-1}(q,p_{N\setminus\{1,j\}})
e^{-\la \bp,w\ra}
\\
=
-\sum_{j=2} ^n
\sum_{p_1=0} ^\infty
  e^{-p_1(w_1+w_j)}
\sum_{\ell=0} ^\infty
\ell e^{-2\ell w_j}
\\
\times
\sum_{q=0} ^\infty e^{-qw_j}
\sum_{p_{N\setminus\{1,j\}}\in\bZ_{+} ^{n-2}}
q\;N_{g,n-1}(q,p_{N\setminus\{1,j\}})
e^{-\la p_{N\setminus\{1,j\}},w_{N\setminus\{1,j\}}\ra}
\\
=
-\sum_{j=2} ^n
\frac{1}{1-e^{-(w_1+w_j)}}\;
\frac{1}
{(e^{w_j}-e^{-w_j})^2}
\left(
-\frac{\partial}{\partial w_j}
\right)
L_{g,n-1}(w_{N\setminus\{1\}}).
\end{multline*}
Summing all contributions, we obtain
\begin{multline*}
\sum_{j=2}^n
\left(
\frac{e^{w_1-w_j}}{1-e^{w_1-w_j}}
+
\frac{1}{1-e^{-(w_1+w_j)}}
\right)
\Bigg[
\frac{1}{(e^{w_1}-e^{-w_1})^2}
\left(
-\frac{\partial}{\partial w_1}
\right)
L_{g,n-1}(w_{N\setminus\{j\}})
\\
-
\frac{1}{(e^{w_j}-e^{-w_j})^2}
\left(
-\frac{\partial}{\partial w_j}
\right)
L_{g,n-1}(w_{N\setminus\{1\}})
\Bigg]
\\
-
\frac{1}{(e^{w_1}-e^{-w_1})^2}
\left(
-\frac{\partial}{\partial w_1}
\right)
\sum_{j=2}^n
L_{g,n-1}(w_{N\setminus\{j\}}).
\end{multline*}
To compute the result of our operation to the fourth and
the fifth lines of
(\ref{eq:integralrecursion}), we note that for any
function $f(q_1,q_2)$ we have
\begin{multline*}
\half \sum_{p_1=0} ^\infty \sum_{0\le q_1+q_2\le p_1}
q_1q_2(p_1-q_1-q_2)e^{-p_1w_1}f(q_1,q_2)
\\
=
\half \sum_{q_1=0}^\infty
\sum_{q_2=0}^\infty
\sum_{\ell=0} ^\infty
2\ell e^{-2\ell w_1}e^{-(q_1+q_2)w_1}q_1q_2f(q_1,q_2)
\\
=
\frac{1}{(e^{w_1}-e^{-w_1})^2}
\left.\frac{\partial^2}{\partial u_1\partial u_2}
\widehat{f}(u_1,u_2)\right|_{u_1=u_2=w_1},
\end{multline*}
where we set $p_1-q_1-q_2=2\ell$, and
$$
\widehat{f}(u_1,u_2)=\sum_{q_1=1} ^\infty \sum_{q_2=1}
^\infty f(q_1,q_2) e^{-(q_1 u_1+q_2 u_2)}.
$$
The reason that  $p_1-q_1-q_2$ is even comes from
the fact that we are summing over $p_N\in\bZ_{+}^n$
subject to $p_1+\cdots+p_n\equiv 0\mod 2$,
while in the fourth line of (\ref{eq:integralrecursion})
contributions vanish unless $q_1+q_2+p_2+\cdots+p_n
\equiv 0\mod 2$. Therefore, we can
restrict the summation over those $p_1,q_1$ and
$q_2$ subject to $p_1\equiv q_1+q_2\mod 2$.
The same condition can be imposed on the
summation for the fifth line of (\ref{eq:integralrecursion}).

Adding all the above, we establish
\begin{multline}
\label{eq:LTNinw}
\frac{\partial}{\partial w_1} L_{g,n}(w_N)
\\
=
\sum_{j=2}^n
\left(
\frac{e^{w_1-w_j}}{1-e^{w_1-w_j}}
+
\frac{1}{1-e^{-(w_1+w_j)}}
\right)
\Bigg[
\frac{1}{(e^{w_1}-e^{-w_1})^2}
\;\frac{\partial}{\partial w_1}
L_{g,n-1}(w_{N\setminus\{j\}})
\\
-
\frac{1}{(e^{w_j}-e^{-w_j})^2}
\;\frac{\partial}{\partial w_j}
L_{g,n-1}(w_{N\setminus\{1\}})
\Bigg]
\\
-
\sum_{j=2}^n
\frac{1}{(e^{w_1}-e^{-w_1})^2}
\;\frac{\partial}{\partial w_1}
L_{g,n-1}(w_{N\setminus\{j\}})
\\
-
\frac{1}{(e^{w_1}-e^{-w_1})^2}\;
\frac{\partial^2}{\partial u_1\partial u_2}
\Bigg[
L_{g-1,n+1}(u_1,u_2,w_{N\setminus\{1\}})
\\
+
\sum_{\substack{g_1+g_2=g\\I\sqcup J=N\setminus\{1\}}}
^{\rm{stable}}
L_{g_1,|I|+1}(u_1,w_I)L_{g_2,|J|+1}(u_2,w_J)
\Bigg]
\Bigg|_{u_1=u_2=w_1}.
\end{multline}
From (\ref{eq:t}) we find
\begin{align*}
\frac{\partial}{\partial w_j}
&=\frac{t_j^2-1}{2}\;\frac{\partial}{\partial t_j}
\\
\frac{1}{(e^{w_j}-e^{-w_j})^2}
&=
\frac{1}{16}\;\frac{(t_j^2-1)^2}{t_j^2}
\\
\frac{e^{w_1-w_j}}{1-e^{w_1-w_j}}
+
\frac{1}{1-e^{-(w_1+w_j)}}
&=
-\frac{t_j(t_1^2-1)}{t_1^2-t_j^2}.
\end{align*}
It is now a straightforward calculation to convert
(\ref{eq:LTNinw}) to
(\ref{eq:LTofN}).
\end{proof}

\section{Initial values}
\label{sect:initial}

In this section we calculate the initial values
$F_{0,3}(t_1,t_2,t_3)$ and $F_{1,1}(t)$.

\begin{figure}[htb]
\centerline{\epsfig{file=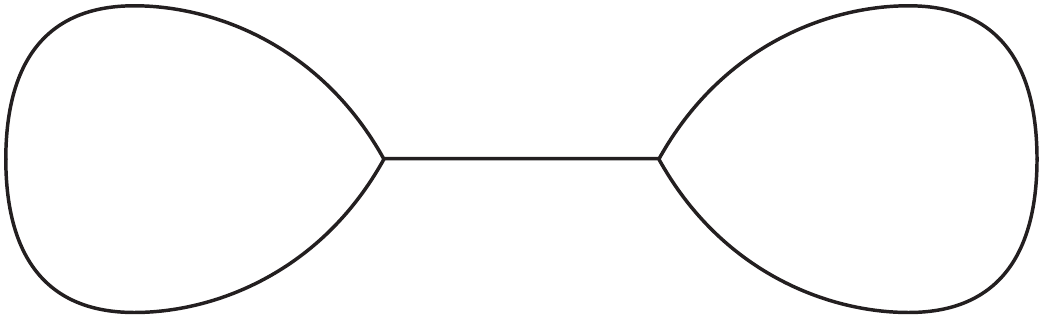, height=0.5in}
\quad
\epsfig{file=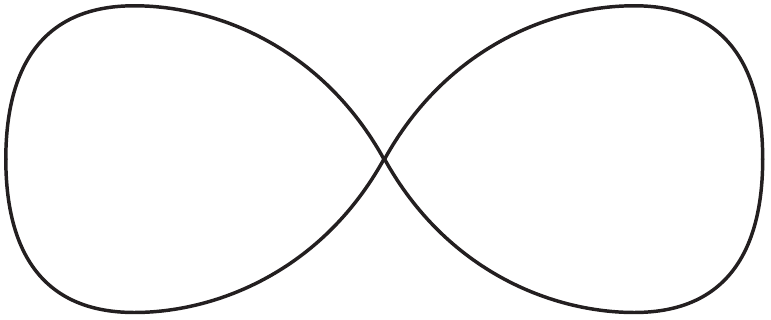, height=0.5in}
\quad
\epsfig{file=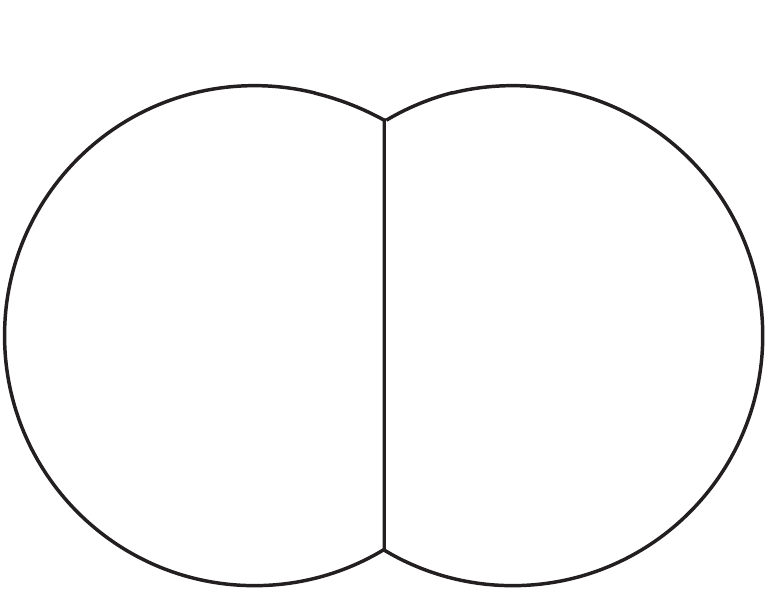, height=0.6in}}
\caption{Three kinds of ribbon graphs of type $(0,3)$.}
\label{fig:03}
\end{figure}

There are three kinds of ribbon graphs of type $(g,n)=(0,3)$ as
listes in Figure~\ref{fig:03}. Each graph has no
nontrivial automorphisms
since every face is fixed.
Therefore, we have
\begin{multline}
\label{eq:F03}
F_{0,3}(t_1,t_2,t_3)
\\
=
(-1)^3
\bigg(
z(t_1,t_1)z(t_1,t_2)z(t_1,t_3)
+z(t_2,t_2)z(t_2,t_1)z(t_2,t_3)
+z(t_3,t_3)z(t_3,t_1)z(t_3,t_2)
\bigg)
\\
+(-1)^2
\bigg(
z(t_1,t_2)z(t_1,t_3)
+z(t_2,t_1)z(t_2,t_3)
+z(t_3,t_1)z(t_3,t_2)
\bigg)
\\
+(-1)^3z(t_1,t_2)z(t_1,t_3)z(t_2,t_3)
\\
=
-\frac{1}{16}(t_1+1)(t_2+1)(t_3+1)
\left(
1+\frac{1}{t_1\;t_2\;t_3}
\right).
\end{multline}
The first line of the right-hand side of (\ref{eq:F03})
corresponds to the dumbbell shape (the left graph
of Figure~\ref{fig:03}), the second line to
the infinity sign (the center graph of Figure~\ref{fig:03}),
and the third line to the right graph of Figure~\ref{fig:03}.

\begin{figure}[htb]
\centerline{
\epsfig{file=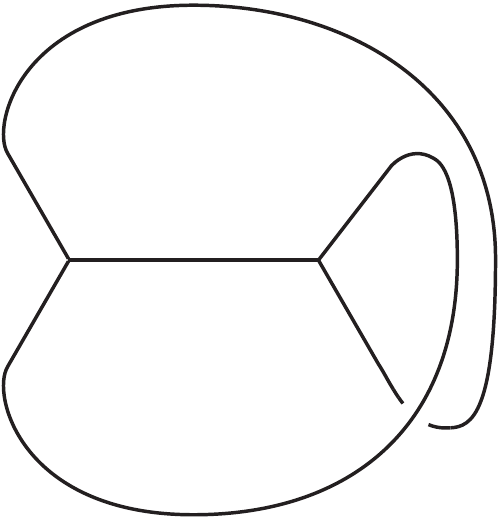, height=0.7in}
\hskip1in
\epsfig{file=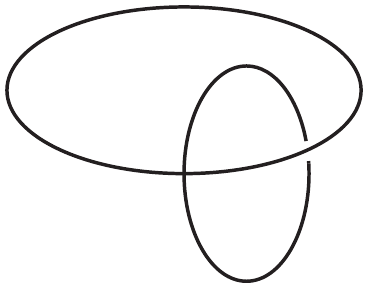, height=0.7in}}
\caption{Two kinds of ribbon graphs of type $(1,1)$.}
\label{fig:11}
\end{figure}

There are two graphs of type $(g,n)=(1,1)$, as shown in
Figure~\ref{fig:11}. The graph on the left has automorphism
group $\bZ/6\bZ$, and the graph on the right has automorphism
group $\bZ/4\bZ$. Thus we have
\begin{equation}
\label{eq:F11}
F_{1,1}(t)
=
\frac{(-1)^3}{6} z(t,t)^3 +\frac{(-1)^2}{4}z(t,t)^2
=
-\frac{1}{384} \;\frac{(t+1)^4}{t^2}\left(t-4+\frac{1}{t}\right).
\end{equation}

\section{Consequences of the differential equation}
\label{sect:consequences}

Since  (\ref{eq:LTofN})
is a differential equation, we need to determine the
initial condition with respect to the variable
$t_1$ in order to uniquely solve it for
$F_{g,n}(t_N)$.
 In this section, we prove Theorem~\ref{thm:main}
 by determining the initial value for the differential
 equation (\ref{eq:LTofN}).

\begin{thm}
\label{thm:unique}
The Poincar\'e polynomial $F_{g,n}(t_N)$ is uniquely
determined by the differential equation
{\rm{(\ref{eq:LTofN})}} and the
vanishing property {\rm{(\ref{eq:Fj})}}.
\end{thm}

\begin{proof}
Suppose we have determined $F_{g,n}(t_N)$ for all
values of $(g,n)$ subject to
$$
0<2g-2+n<m-1
$$
for a given
$m\ge 2$. Take any $(g,n)$ such that $2g-2+n=m$.
Then  (\ref{eq:LTofN})
determines
$\frac{\partial}{\partial t_1}F_{g,n}(t_N)$.
We denote by $r(t_N)$ the right-hand side of
 (\ref{eq:LTofN}),
and \emph{define}
\begin{equation}
\label{eq:integral}
F_{g,n}(t_N) = \int_{-1}^{t_1} r(t_N) dt_1.
\end{equation}
The lower bound is chosen so that (\ref{eq:Fj}) holds.
Since $F_{g,n}(-1,t_2,\dots,t_n)=0$ as a function
in $t_{N\setminus\{1\}}$, there is no room to add
any function in $t_{N\setminus\{1\}}$ to the
right-hand side of (\ref{eq:integral}).
We have thus uniquely determined $F_{g,n}(t_N)$.
This completes the proof.
\end{proof}

Since formula  (\ref{eq:integral}) is  (\ref{eq:main}),
we have thus proved Theorem~\ref{thm:main}.

The definition of the Poincar\'e polynomial
(\ref{eq:Poincare}) contains a factor like
$\frac{1}{t_i+t_j}$. Surprisingly,
$F_{g,n}(t_N)$  is indeed a \emph{Laurent polynomial}.

\begin{thm}
\label{thm:Laurent}
The Poincar\'e polynomial $F_{g,n}(t_N)$ is a Laurent
polynomial in $t_1,t_2,\dots, t_n$ of degree $3(2g-2+n)$.
Moreover, every monomial appearing in $F_{g,n}(t_N)$
contains only an odd power of each $t_j$.
\end{thm}

\begin{proof}
Here again suppose the statement is true for all
values of $(g,n)$ subject to
$$
0<2g-2+n<m-1
$$
for a given
$m\ge 2$. Take an arbitrary $(g,n)$ such that
$2g-2+n=m$. Let $r(t_N)$ denote the right-hand side
of (\ref{eq:LTofN}).
There are two issues we need to address. The first one
is division by $(t_1^2-t_j^2)$ in the first line of
$r(t_N)$, since $\frac{1}{t_1^2-t_j^2}$ is not a
Laurent polynomial. The second issue is the
integration (\ref{eq:integral}), which could produce
logarithmic terms.

\begin{lem}
Consider a Laurent polynomial in one variable $f(x)$ that
contains only odd powers of $x$. Then
\begin{equation}
\label{eq:fx}
\frac{y}{x^2-y^2}
\left(
\frac{(x^2-1)^3}{x^2}\frac{\partial}{\partial x}\;f(x)
-
\frac{(y^2-1)^3}{y^2}\frac{\partial}{\partial y}\;f(y)
\right)
\end{equation}
is a Laurent polynomial in $x$ and $y$ such that each
monomial contains only an even power of $x$ and an odd power
of $y$.
\end{lem}

\noindent
If $h(x)$ is a Laurent polynomial in $x^2$, then
$\frac{h(x)-h(y)}{x^2-y^2}$ is a Laurent polynomial
in $x^2$ and $y^2$. Therefore,
$$
\frac{1}{x^2-y^2}
\left(
\frac{(x^2-1)^3}{x^2}\frac{\partial}{\partial x}\;f(x)
-
\frac{(y^2-1)^3}{y^2}\frac{\partial}{\partial y}\;f(y)
\right)
$$
is a Laurent polynomial in $x^2$ and $y^2$. This proves
the lemma.

Thus we know that $r(t_N)$ is a Laurent polynomial
in $t_1,\dots, t_n$ such that each monomial contains
an even power of $t_1$ and an odd powers of $t_j$ for
every $j>1$. Therefore,
$$
F_{g,n}(t_N) = \int_{-1}^{t_1} r(t_N)dt_1
$$
is a Laurent polynomial in $t_1, \dots, t_n$ such that every
monomial term
contains only an odd power of each $t_j$. This completes the
proof of the theorem.
\end{proof}

Based on the work \cite{BCSW},
 it is noted in \cite{CMS} that the symmetric
 homogeneous polynomial
in $t_1,\dots,t_n$ consisting of
the leading terms of
$$
2^{5g-5+2n}\;\frac{\partial^n}{\partial t_1\cdots\partial t_n}
 F_{g,n}(t_N)
 $$
 is the
generating function of the $\psi$-class intersection numbers
on the Deligne-Mumford stack $\Mbar_{g,n}$
considered in \cite{DVV, K1992,W1991}, and that
the restriction of the recursion (\ref{eq:LTofN})
to the leading terms, after taking the differentiation with respect
to $t_2, \dots, t_n$, is equivalent to the Virasoro constraint
condition of the $\psi$-class intersection numbers.
This proves Theorem~\ref{thm:appl}.

Although we do not utilize the following fact in this paper,
we note that the Laurent polynomial $F_{g,n}(t_N)$
is invariant under the coordinate change
$t_j\longmapsto\frac{1}{t_j}$. This is because
$$
z(t_i,t_j) = z\!\left(\frac{1}{t_i},\frac{1}{t_j}\right).
$$

\begin{prop}
The Poincar\'e polynomial is invariant under the
transformation $t_j\longmapsto\frac{1}{t_j}$
$$
F_{g,n}(t_1,t_2,\dots,t_n) = F_{g,n}\!
\left(
\frac{1}{t_1},\frac{1}{t_2},\dots,\frac{1}{t_n}
\right).
$$
\end{prop}

\begin{appendix}
\section{Examples}
\label{sect:examples}
We record a few examples of the Poincar\'e polynomials
here.

\begin{multline}
\label{eq:F04}
F_{0,4}(t_1,t_2,t_3,t_4)=
\frac{1}{2^8}\;(t_1+1)(t_2+1)(t_3+1)(t_4+1)
\\
\times
\left[
\sum_{j=1}^4 t_j^2
-
\sum_{j=1}^4 t_j
-5
-
\sum_{i<j}\frac{1}{t_it_j}
+
\frac{1}{t_1t_2t_3t_4}
\left(-5
-\sum_{j=1}^4 \frac{1}{t_j}
+\sum_{j=1} ^4 \frac{1}{t_j^2}
\right)
\right].
\end{multline}

\begin{multline}
\label{eq:F12}
F_{1,2}(t_1,t_2)=\frac{1}{2^{11}}\;
\Bigg[
t_1^5 t_2+t_1t_2^5+\frac{t_1^3\;t_2^3}{3}+t_1^5+t_2^5
-6(t_1^3t_2+t_1 t_2^3)
-\frac{17}{3}\;(t_1^3+t_2^3)
\\
+27\; t_1t_2+26(t_1+t_2) +\frac{128}{3}
+4\left(\frac{t_1}{t_2}+\frac{t_2}{t_1}\right)
+26\left(\frac{1}{t_1}+\frac{1}{t_2}\right)
+\frac{27}{t_1t_2}
\\
-\frac{17}{3}\left(\frac{1}{t_1^3}+\frac{1}{t_2^3}\right)
-6\left(\frac{1}{t_1t_2^3}+\frac{1}{t_1^3t_2}\right)
+\left(\frac{1}{t_1^5}+\frac{1}{t_2^5}\right)
+\frac{1}{3}\;\frac{1}{t_1^3\;t_2^3}
+\left(\frac{1}{t_1t_2^5}+\frac{1}{t_1^5t_2}\right)
\Bigg].
\end{multline}

\begin{multline}
\label{eq:F21}
F_{2,1}(t) =
-\frac{1}{2^{19}}\;\frac{(t+1)^8}{t^4}
\Bigg(
\frac{35}{3}\;t^{5} -\frac{280}{3}\;t^4+333\;t^3
-704\;t^2
+\frac{5018}{5}\;t
\\
-\frac{5424}{5}
+\frac{5018}{5}\;t^{-1}
-704\;t^{-2}
+333\;t^{-3}
-\frac{280}{3}\;t^{-4}+\frac{35}{3}\; t^{-5}
\Bigg)
\\
=
-\frac{35}{6} \; z^9+\frac{105}{4}\; z^8
-\frac{93}{2}\; z^7 +\frac{161}{4}\; z^6
-\frac{84}{5}\;z^5 + \frac{21}{8}\; z^4,
\end{multline}
where
$z$ is defined by (\ref{eq:ztt}).

\begin{multline}
\label{eq:F31}
F_{3,1}(t) =
-\frac{1}{2^{30}}\;\frac{(t+1)^{12}}{t^{6}}
\Bigg(
\frac{5005}{3}\;t^{9}
-20020\;t^{8}
+112343\;t^{7}
-\frac{1181488}{3}\; t^{6}
\\
+975692\; t^{5}
-1842448\; t^{4}
+\frac{25312028}{9}\; t^{3}
-\frac{10959056}{3}\;t^{2}
+\frac{88361050}{21}\;t
\\
-\frac{277329032}{63}
+\frac{88361050}{21}\;t^{-1}
-\frac{10959056}{3}\;t^{-2}
+\frac{25312028}{9}\; t^{-3}
-1842448\; t^{-4}
\\
+975692\; t^{-5}
-\frac{1181488}{3}\; t^{-6}
+112343\;t^{-7}
-20020\;t^{-8}+\frac{5005}{3}\; t^{-9}
\Bigg)
\\
=
-\frac{5005}{3}\;z^{15}
+\frac{25025}{2}\; z^{14}
- 41118\; z^{13}
+\frac{929929}{12}\; z^{12}
- \frac{183955}{2}\; z^{11}
\\
+ \frac{283767}{4}\; z^{10}
- \frac{317735}{9} \;z^9
+10813\; z^8
- \frac{25443}{14}\; z^7 + \frac{495}{4} z^6.
\end{multline}

\end{appendix}

\begin{ack} The authors are indebted to
Kevin Chapman for providing them with the computational
results reported in
\cite{CMS}. 
During the preparation of this paper,
M.M.\ received support from the American Institute of
Mathematics,
the National Science Foundation, and
Universiti Teknologi Malaysia. The paper was
completed while he was visiting  Tsinghua University in Beijing.
He is particularly grateful to Jian Zhou for his 
hospitality, stimulating discussions
and useful comments on this work.
The research of M.P.\ was supported
 by the University of California, Davis, 
 and the University of Wisconsin-Eau Claire.
\end{ack}


\providecommand{\bysame}{\leavevmode\hbox to3em{\hrulefill}\thinspace}

\bibliographystyle{amsplain}

\end{document}